\documentclass[a4paper,12pt]{amsart}
\usepackage[utf8]{inputenc}
\usepackage{mathtools}
\usepackage{amsmath}
\usepackage{amssymb}
\usepackage{graphicx}
\usepackage{changepage}
\graphicspath{ {images/} }
\usepackage{mathrsfs}
\usepackage{listings}
\usepackage{amsthm}
\usepackage[toc,page]{appendix}
\usepackage[margin=1.1in, headsep=10pt]{geometry}
\usepackage{framed,enumitem}
\usepackage{amsaddr}
\usepackage{dsfont}

\usepackage[colorlinks = true,
linkcolor = blue,
urlcolor  = blue,
citecolor = blue,
anchorcolor = blue]{hyperref}

\renewenvironment{abstract}
{\small
	\begin{center}
		\bfseries \abstractname\vspace{-.5em}\vspace{0pt}
	\end{center}
	\list{}{%
		\setlength{\leftmargin}{14mm}
		\setlength{\rightmargin}{\leftmargin}%
	}%
	\item\relax}
{\endlist}

\newtheorem{thm}{Theorem}[section]
\newtheorem{defn}[thm]{Definition}
\newtheorem{cor}[thm]{Corollary}

\newtheorem{prop}[thm]{Proposition}
\newtheorem{lem}[thm]{Lemma}
\newtheorem{rk}[thm]{Remark}

\title{Convergence of ASEP to KPZ with basic coupling of the dynamics}

\author{Shalin Parekh}
\address{Columbia University}
\email{sp3577@columbia.edu}

\date{\today}

\begin{document}

\maketitle

\begin{abstract}\textbf{}
\\ 
We prove an extension of a seminal result of Bertini and Giacomin. Namely we consider weakly asymmetric exclusion processes with several distinct initial data simultaneously, then run according to the basic coupling, and we show joint convergence to the solution of the KPZ equation with the same driving noise in the limiting equation. Along the way, we analyze fine properties of nontrivially coupled solutions-in-law of KPZ-type equations.
\end{abstract}

\section{Introduction and context}

Interacting particle systems on the integer lattice have been a popular area of research in recent years. Of particular interest is the asymmetric simple exclusion process (ASEP), which was introduced by Spitzer \cite{Spi70} and subsequently generalized and explored in many works. ASEP is a Feller process on $\{0,1\}^\Bbb Z$ in which one starts with an initial configuration on $\Bbb Z$ consisting of some particles (1's) and some empty sites (0's), and the evolution of the dynamics can be described by having the particles independently perform asymmetric nearest-neighbor (continuous-time) random walks on $\Bbb Z$, but with jumps suppressed whenever one particle tries to jump onto another one. This hard-core repulsion effect between the particles makes the system physically interesting but also mathematically difficult to analyze.
\\
\\
In a seminal paper, Bertini and Giacomin \cite{BG97} showed that under a certain fluctuation regime and specific scaling of the jump parameters, the fluctuations of ASEP are described by a nonlinear stochastic partial differential equation called the Kardar-Parisi-Zhang (KPZ) equation: 
\begin{equation*}\partial_t h(t,x) = \partial_x^2 h(t,x) + (\partial_x h(t,x))^2 +\xi(t,x),
\end{equation*}where $\xi$ is Gaussian space-time white noise, specified by the formal covariance function $\Bbb E[\xi(t,x) \xi(s,y)] = \delta(t-s)\delta(x-y).$ More specifically, Bertini and Giacomin considered ASEP where the right jump rate for each particle equals $1+\epsilon^{1/2}$ and the left jump rate equals $1$. They consider initial data which is ``near stationarity" in a certain precise way. They then define a discrete height function $h^{\epsilon}_t(x)$ ($t \ge 0,x\in \Bbb Z)$ as follows: $h^\epsilon_t(0)$ is the number of particles up to time $t$ that have passed from $0$ to $1$, minus the number of particles that have passed from $1$ to $0$. Then $h^\epsilon_t(x)$ equals $h^\epsilon_t(0)$, plus the number of particles at time $t$ which are between $0$ and $x$ (inclusive), minus the number of vacant sites between $0$ and $x$ (understood to be linearly interpolated when $x$ is not an integer). They then prove that $\epsilon^{1/2}h^{\epsilon}(\epsilon^{-2}t,\epsilon^{-1}x) - \epsilon^{-1}t-t/24$ converges as $\epsilon \to 0$, to the Hopf-Cole solution of the KPZ equation (see Theorem \ref{bg97} for a precise version).
\\
\\
The result was striking because it was one of the first examples of a particle system in a regime that exhibited non-Gaussian fluctuation behavior, and it was one of the works that paved the way to the field of KPZ universality for random growth models, see the survey \cite{Cor12} as well as subsequent recent work on particle systems that built on, generalized, or was inspired by the work of Bertini and Giacomin, e.g. \cite{ACQ11, BQS11, GJ14, DT16, CT17, CST18, CS18, Yang} just to name a few.
\\
\\
The main goal of the present work is to prove that in the fluctuation regime of \cite{BG97}, if one starts with \textit{two or more} different initial data, and then one runs the particle system according to the \textit{same} dynamics, then convergence to KPZ holds \textit{jointly} with the same realization of the noise appearing in the limiting equation. When we refer to the ``same dynamics," we are referring to the so-called \textit{basic coupling}, a natural and important object that appears in many contexts when dealing with exclusion systems, e.g. in providing a full description of the ergodic theory of exclusion processes, see \cite{Lig74, Lig05, GS10}. This basic coupling is described as follows: for each pair of sites $(x,y) \in \Bbb Z^2$ if a particle from both systems is present at $x$, and if a particle from one system jumps from $x$ to $y$, then a particle from the other system also jumps from $x$ to $y$ at the same time assuming the target site is not blocked in the other configuration. This coupling can be constructed by the so-called ``graphical construction" of ASEP, which randomly assigns directed arrows to each bond in $\Bbb Z$ according to independent Poisson point processes, see e.g. \cite{Har72,Sep05}. A slightly more general definition and construction of the coupling is given in Subsection 3.1. 
\\
\\
With this setup, let us now state our main result. We will say that a sequence $(\epsilon^{1/2}h^{1,\epsilon}(0,\epsilon^{-1}x),\epsilon^{1/2}h^{2,\epsilon}(0,\epsilon^{-1}x))$ of initial data for the height functions is \textit{admissible} if it converges (jointly) in law to some limiting pair of height functions, and if it is tight in the sense that the $L^p$ moments of its absolute value and of its spatial differences can be bounded via a Kolmogorov-Chentsov criterion with a sublinear growth rate at infinity. The precise assumption is given in Section 3, see Theorem \ref{mr2}.

\begin{thm}
\label{mmrr}
Consider the weakly asymmetric scaling of ASEP from \cite{BG97}. Let $(h_0^{1,\epsilon},h_0^{2,\epsilon})$ be an admissible sequence of initial data. Evolve the corresponding height functions $h^{1,\epsilon}(t,x), h^{2,\epsilon}(t,x)$ according to the basic coupling described above. Then $(h^{1,\epsilon}, h^{2,\epsilon})$ converge jointly as $\epsilon\to 0$ to the solution of the KPZ equation driven by the same noise.
\end{thm}

This result will be stated more precisely and proved as Theorems \ref{mr} and \ref{mr2} below. The main difficulty lies in the fact that for interacting particle systems such as ASEP, some of the jumps are suppressed due to the fact that particles are not allowed to jump onto other particles. To prove the result, one may convince themselves that it is somehow necessary to keep track of the noise as well as the height profile in the limit, not just the latter. At first glance, one might try to show that $(h^{1,\epsilon},h^{2,\epsilon},\xi^\epsilon)$ converge jointly as $\epsilon\to 0$ to $(h^1,h^2,\xi)$, where $h^{i,\epsilon}$ are the rescaled, renormalized, and basically coupled height profile as described earlier, where $\xi^\epsilon$ keeps track of the Poisson clocks which excite the particles to jump, and where $h^1,h^2$ both solve the KPZ equation with the same noise $\xi$. Unfortunately, this approach is bound to fail because approximately half of the Poisson clocks go unused by the system due to suppressed jumps. In reality, the ``correct" discretization of the noise consists of only those Poisson clocks which are used by the system. But this depends intricately on the initial data of the system. In other words, there is no natural choice of $\xi^\epsilon$ above: there is always a $\xi^{1,\epsilon}$ associated with $h^{1,\epsilon}$ and likewise there is $\xi^{2,\epsilon}$ for $h^{2,\epsilon}$. And the primary technical task is to relate the $\xi^{i,\epsilon}$ for $i=1,2$, in particular to prove that these converge to the same noise in the limit. So one runs into a vicious cycle which creates a difficulty in the arguments.
\\
\\
In terms of applications of our theorem, one can recover a few results about how joint solutions of KPZ behave when run according to the same noise $\xi$. Here is just one example: consider the stochastic Burgers equation $$\partial_t u = \partial_x^2 u + \partial_x(u^2) +\partial_x\xi,$$ which is formally related to the KPZ equation by $u=\partial_x h$, and indeed one can define the solution this way interpreted in terms of distributions. Consider two solutions $u^1,u^2$ of stochastic Burgers driven by the same realization of $\xi,$ started from two initial data $u_0^1,u_0^2$ respectively. Suppose that the initial data are ordered, i.e., $u_0^1\leq u_0^2$ deterministically in the sense that $u_0^2-u_0^1$ is a positive Borel measure. Then Theorem \ref{mmrr} implies that $u^1(t,\bullet)\leq u^2(t,\bullet)$ almost surely for all $t,$ again interpreted in the sense that the difference is a positive measure. In other words, the KPZ dynamics preserve the property that the difference of height functions is nondecreasing. This is because the ordering is preserved at the level of the particle systems, see \eqref{A} below. This result can very likely be proved using other methods as well, for instance proving the result first for smooth noises $\xi$ (see for instance Section 3 of \cite{DGR21}) and then using an approximation of space-time white noise by spatial mollifications and using the fact that the desired result is stable under limits and that the associated solutions converge after height renormalization (see e.g. \cite{PR19}). One advantage in our discretization via ASEP is that the result is already obvious at the level of the particle system without using PDE techniques. 
\\
\\
The input to proving our main theorem will require two steps. First we will prove a result (Theorem \ref{mon1} below) about nontrivially coupled KPZ's, which says that two solutions-in-law of the KPZ equation with the property that their difference has zero quadratic variation in the $x$ variable must in fact be driven by the same noise. This result may be of independent interest, and it will be the main tool to identify joint limit points of the coupled height functions. The other tool we will use is the monotonicity and attractivity properties of ASEP and related systems. It should be noted that our methods are easily generalizable to other types of basically coupled systems that satisfy these properties as well, such as joint convergence of the \textit{symmetric} simple exclusion process to the Edwards-Wilkinson fixed point as well as higher-spin processes for which KPZ fluctuations are known, such as ASEP($q,J$) \cite{CST18}. We discuss the latter model in Subsection 3.5. 
\\
\\
\textbf{Outline:} In Section 2, we prove a result about coupled solutions-in-law of the KPZ equation. In Section 3 we prove Theorem \ref{mmrr}. Subsection 3.1 introduces the basic coupling model and the notations, Subsection 3.2 describes the result of Bertini-Giacomin in some detail, Subsection 3.3 contains the proof of our main result in the case of deterministic initial data (Theorem \ref{mr}) and then Subsection 3.4 contains the main result for randomized initial data, Theorem \ref{mr2}. Subsection 3.5 then includes a discussion of how to generalize our argument to more complex models.
\\
\\
\textbf{Acknowledgements:} We thank Ivan Corwin for suggesting the problem. The author was partially supported by the Fernholz
Foundation’s “Summer Minerva Fellows” program, as well as summer support from
Ivan Corwin’s NSF grant DMS:1811143.

\section{A result about nontrivially coupled KPZ's}

To prove the main result, we use a continuum apparatus which allows us to efficiently identify joint limit points of the coupled particle system. To formulate our result we consider a slightly more general version of the KPZ equation with a parameter $\lambda \in \Bbb R$:
\begin{equation}\label{KPZ}\tag{KPZ}\partial_t h(t,x) = \partial_x^2 h(t,x) + \lambda (\partial_x h(t,x))^2 +\xi(t,x).
\end{equation} We use the notion of the so-called Hopf-Cole solution, which uses the fact that if $h$ solves \eqref{KPZ} then $Z:=e^{\lambda h}$ solves the multiplicative noise equation given by $\partial_tZ = \partial_x^2 Z+ \lambda Z\xi$ which actually turns out to be well-posed using classical methods from \cite{Wal86}. To make this rigorous, one formulates all of this using the Duhamel principle:

\begin{defn}[Hopf-Cole solution]\label{hopf}
Let $P(T,X) = \frac1{\sqrt{2\pi T}} e^{-X^2/2T},$ and let $\xi$ denote a space-time white noise on some probability space $(\Omega, \mathcal F, \Bbb P)$. Let $Z_0$ denote some (random) Borel measure on $\Bbb R$. We say that a continuous space-time process $h = (h(T,X))_{T > 0, X \in \Bbb R}$ is a solution of \eqref{KPZ} if $\Bbb P$-almost surely, for every $T>0$ and $X \in \Bbb R$, the process $Z(T,X):=e^{\lambda h(T,X)}$ satisfies the identity $$ Z(T,X) = \int_\Bbb R P(T, X-Y)Z_0(dY)+\lambda \int_0^T\int_\Bbb R P(T-S,X-Y) Z(S,Y) \xi(dS,dY),$$ where the integral against the white noise is meant to be interpreted in the Itô-Walsh sense \cite{Wal86}. 
\end{defn}

Next we will define the class of initial data for which our apparatus will be applicable. This class of functions will also be used extensively in later sections of the paper.

\begin{defn}\label{hold}
Let $\alpha,\delta \in (0,1).$ A function $f:\Bbb R\to \Bbb R$ is said to be in the $\delta$-weighted $\alpha$-Hölder space $\mathscr{C}^{\alpha}_{\delta}(\Bbb R)$ if $$\sup_{x \in \Bbb R} \frac{|f(x)|}{(1+|x|)^\delta}+\sup_{\substack{x,y\in \Bbb R\\|x-y|\leq 1}} \frac{|f(x)-f(y)|}{(1+|x|)^{\delta} |x-y|^{\alpha}} <\infty.$$
We turn $\mathscr C_\delta^\alpha$ into a Banach space by defining the norm of $f$ to be the above quantity.
\end{defn}

We are going to prove a result which roughly says that if we have two space-time processes defined on the same probability space, each solving \eqref{KPZ} in law, not necessarily driven by the same noise but their difference satisfies some specific nontrivial deterministic condition, then the two noises must in fact be the same.

\begin{thm}\label{mon1}
Suppose we have two standard space-time white noises $\xi^1,\xi^2$ coupled onto the same probability space. Suppose furthermore that they satisfy the following conditions:
\begin{enumerate}
    
    \item $\Bbb E[(\xi^1,f) (\xi^2,g)] = 0$ for all $f,g\in L^2(\Bbb R_+\times \Bbb R)$ which have disjoint supports.
    
    \item For every $t> 0$ the spatial process $h^2(t,\cdot) - h^1(t,\cdot)$ has a.s. finite $p$-variation for some $p<2$, where $h^i$ is a solution of $\partial_t h^i = \partial_x^2 h^i + \lambda_i(\partial_xh^i)^2+ \xi^i,$ for $i=1,2.$ Here $\lambda_1,\lambda_2 \in \Bbb R$ and furthermore we assume that the initial data $h^i(0,\cdot) \in \mathscr C_\delta^\alpha$ for some $\alpha,\delta \in (0,1).$
\end{enumerate}
Then $\xi^1 = \xi^2$.
\end{thm}

We remark that the two noises are \textbf{not} assumed to be \textit{jointly} Gaussian. This will be important while applying the theorem later.

\begin{proof}
Define a bilinear form $I$ on $L^2(\Bbb R_+ \times \Bbb R)$ by $I(\phi,\psi) := \Bbb E[(\xi^1,\phi) (\xi^2,\psi)].$ By Cauchy-Schwarz
$$|I(\phi,\psi)| \leq \Bbb E[(\xi^1,\phi)^2]^{1/2}\Bbb E[(\xi^2,\psi)^2]^{1/2} = \|\phi\|_2\|\psi\|_2.$$
Thus $I$ is bounded, so by Riesz representation theorem there exists some bounded operator $A: L^2(\Bbb R_+ \times \Bbb R) \to L^2(\Bbb R_+ \times \Bbb R)$ such that $I(\phi,\psi) = \langle \phi, A\psi\rangle_{L^2(\Bbb R_+ \times \Bbb R)}$ and $\|A\|\leq 1$.
\\
\\
Note that $\langle \phi, A\psi \rangle_{L^2(\Bbb R_+\times \Bbb R)} =0$ whenever $\phi$ and $\psi$ have disjoint supports on $\Bbb R_+\times \Bbb R$. The reader may show that any operator on an $L^2$ space (associated with a sigma finite measure) which satisfies this property is necessarily a multiplication operator. Thus there exists some $\mathbf v \in L^{\infty} (\Bbb R_+ \times \Bbb R)$ such that $A\phi = \mathbf v \cdot \phi$ for all $\phi \in L^2 (\Bbb R_+ \times \Bbb R)$. Note that $\|\mathbf v\|_{L^{\infty}(\Bbb R_+ \times \Bbb R)} = \|A\| \le 1$.
\\
\\
We have shown that if $\phi, \psi \in L^2(\Bbb R_+ \times \Bbb R)$ then 
\begin{equation}\Bbb E[(\xi^1,\phi) (\xi^2,\psi)] = \int_{\Bbb R_+\times \Bbb R} \phi(t,x)\psi(t,x)\mathbf v(t,x)dt\;dx,\label{v}
\end{equation}
where $|\mathbf v(t,x)|\le 1$ a.e. Note that $\xi^1, \xi^2$ have \textit{not} been shown or assumed to be jointly Gaussian. Our goal is now to show that $\mathbf v=1$ a.e. on $\Bbb R_+ \times \Bbb R$.
\\
\\
For $i=1,2$ we define $X^i(t,x)$ for $t \ge 0$ and $x \in \Bbb R$ as the solution of the linear SPDE $$\partial_t X^i = \partial_x^2 X^i + \xi^i,$$ with $X^i(0,x)=0$. Letting $h^i$ be as in the theorem statement, we can write $h^i(t,x) = X^i(t,x) + v^i(t,x),$ where $h^i_0(x) = h^i(0,x)$ and $v^i$ is a remainder term which is locally Holder continuous of exponent strictly greater than $1/2$ in the spatial variable. For the KPZ equation on the \textit{circle} $\Bbb T$, the existence of such a remainder term $v^i$ was first proved as Theorem 1.10 in \cite{Hai13} using a preliminary version of the theory of regularity structures. We believe that the result on the \textit{full line} $\Bbb R$ (which is what we need) can also be proved using regularity structures, however it has not been done in the literature thus far (in the introduction of \cite{HL18}, there is a discussion of the difficulties involved with making direct sense of the full-line KPZ equation). However, the full line result can instead be deduced from Definitions 3.2, 3.3, and Theorem 3.19 in \cite{PR19} which uses the theory of paracontrolled products to make direct sense of the full-line KPZ equation. The fact the notion of solution used there coincides with the Hopf-Cole solution also follows Theorem 3.19 there. However, that theorem assumes that the initial data lie in $\mathscr C^\alpha_\delta$ (see Assumptions 3.7 and Remark 3.8 in \cite{PR19}) which is the only reason we have assumed such a restriction on the class of initial data in this theorem and in later parts of this paper. This assumption can likely be relaxed, but it does not seem to have been done in the literature thus far.
\\
\\
Now let $Y:= X^2-X^1$. Then $$Y(t,x) = \big[h^2(t,x) - h^1(t,x)\big] + \big[ v^1(t,x) - v^2(t,x) \big].$$
By assumption, for each fixed $t>0$, each of the two terms in the square brackets have a.s. finite $p$-variation in the $x$ variable, for some $p<2$ (since the $v^i$ are spatially Holder continuous of exponent strictly greater than $1/2$). Thus, $Y$ has a.s. finite $p$-variation in the $x$ variable.
\\
\\
Define a sequence of random variables 
$$ Q_N(t):= \sum_{k=1}^{2^N} \big(Y(t,2^{-N}k) - Y(t,2^{-N}(k+1))\big)^2.$$ Since $Y$ is of finite $p$-variation in the $x$ variable with $p<2$, and since $Q_N$ is approximating the \textit{quadratic} variation, it follows that $Q_N(t) \to 0 $ almost surely as $N \to \infty$. We claim that $\Bbb E[Q_N(t)] \to 0$ as well. To prove this, it suffices to show that $\sup_N\Bbb E[Q_N(t)^q]<\infty$
for some $q>1$, as that implies uniform integrability. To show this uniform $L^q$ bound, note that $(a-b)^2 \leq 2a^2+2b^2$ for all $a,b$, and  recall that $Y=X^1-X^2$. Therefore 
$$Q_N(t) \leq 2 \sum_{k=1}^{2^N} \sum_{i=1,2} \big(X^i(t,2^{-N}k) - X^i(t,2^{-N}(k+1))\big)^2,$$ 
so that 
\begin{align*}\Bbb E[Q_N(t)^q] &\leq 2 \cdot 2^{(N+1)(q-1)} \sum_{k=1}^{2^N}\sum_{i=1,2}\Bbb E[\big|X^i(t,2^{-N}k) - X^i(t,2^{-N}(k+1))\big|^{2q}] \\&= 4 \cdot 2^{(N+1)(q-1)} \sum_{k=1}^{2^N} \Bbb E[\big|X^1(t,2^{-N}k) - X^1(t,2^{-N}(k+1))\big|^{2q}]\\ &= 2^{Nq+q+1} \Bbb E[\big|X^1(t,2^{-N}) - X^1(t,0))\big|^{2q}] \\ &\leq C_q \cdot 2^{Nq} \Bbb E\big[\big(X^1(t,2^{-N}) - X^1(t,0))\big)^2\big]^q
\end{align*}
Here the first inequality is obtained by using Hölder (or Jensen) on the double sum from the previous expression, which allows us to bring the $q^{th}$ power inside the sum at a cost of an extra factor $2^{(N+1)(q-1)}$. The equality in the second line holds because $X^1$ and $X^2$ have the same distribution as space-time fields, so the sum over $i=1,2$ simply doubles the expectation of the $i=1$ case. The equality in the third line holds because $X^1$ is stationary in $x$ (recall that it was started from zero initial data) and thus the terms in the sum do not depend on $k$. In the last inequality $C_q$ is a constant depending on $q$ but not $N$, and it holds because $X^1(t,2^{-N}) - X^1(t,0)$ has a centered normal distribution, and thus satisfies the standard ``reverse Jensen" bounds. With all of this in place, we just need to show that $\Bbb E[\big(X^1(t,2^{-N}) - X^1(t,0))\big)^2] \leq C2^{-N}$. But this is standard, see for instance Section 2.3 of \cite{Hai09} for a precise computation which shows that $\Bbb E[(X^1(t,x) - X^1(t,y))^2] \leq C|x-y|$ where $C$ is independent of $t,x,y$. Thus we have shown that $\Bbb E[Q_N(t)] \to 0$ as $N \to \infty$.
\\
\\
Recall that the goal is to show that $\mathbf v=1$ a.e. To do this, we will now compute $\Bbb E[Q_N(t)]$ in a different manner using $\mathbf v$. We can write $Y$ in mild form as $Y_t =p * (\xi^1-\xi^2),$ where $*$ denotes space-time convolution and $p$ is the standard heat kernel as always. Thus, by using \eqref{v} we see that
\begin{equation}\label{EQ}\Bbb E[Q_N(t)] = 2\int_0^t \int_{\Bbb R} \bigg[ \sum_{k=1}^{2^N} \big( p_{t-s} (x_k-z)-p_{t-s}(x_{k+1}-z) \big)^2\bigg] \big(1-\mathbf v(s,z) \big) dz\; ds,\end{equation}
where $x_k := k\cdot 2^{-N}.$ Now we will show that the limit of this quantity is strictly positive for some $t > 0$ unless $1-\mathbf v$ vanishes a.e. on $\Bbb R_+ \times \Bbb R$. The only major difficulty is that $\mathbf v$ has $L^{\infty}$ regularity at best, and the part of the integrand in the square brackets is converging weakly as $N \to \infty$ to a measure which is singular with respect to 2D Lebesgue measure, so taking a limit of the above integral is somewhat tricky and will involve using the Lebesgue differentiation theorem from measure theory. Define 
$$\alpha:= \min_{\substack{t \in [1,2]\\ x\in [1,3]}} (p_t(x-1)-p_t(x+1)) > 0.$$
By using the relation $p_t(x) = \epsilon^{-1}p_{\epsilon^{-2} t}(\epsilon^{-1}x)$, valid for all $\epsilon,t>0$ and $x\in \Bbb R$, we see that $$\min_{\substack{t \in [\epsilon^2, 2\epsilon^2]\\ x \in [\epsilon ,3\epsilon ]}} (p_t(x-\epsilon )-p_t(x+\epsilon)) = \epsilon^{-1}\alpha,\;\;\;\; \text{ for all } \epsilon>0. $$
Thus $\big(p_t(x-\epsilon)-p_t(x+\epsilon)\big)^2 \ge \alpha^2 \epsilon^{-2} \big(1_{[\epsilon,3\epsilon]}+1_{[-3\epsilon,-\epsilon]})(x), $ for all $t \in [\epsilon^2, 2\epsilon^2].$ Taking $\epsilon= 2^{-N-1}$, we see that 
\begin{equation}\label{EQ2}\sum_{k=1}^{2^N} \big( p_{t-s} (x_k-z)-p_{t-s}(x_{k+1}-z) \big)^2 \geq 4^{N+1} \alpha^2, \end{equation}
whenever $s \in [t-2\cdot 4^{-N-1}, t-4^{-N-1}]$ and $z \in [0,1]$. Here $x_k=2^{-N}k$ as always.
\\
\\
For $t \ge 0$ define $u(t): = \int_0^1 \int_0^t (1-\mathbf v(s,z)) ds dz$. By combining \eqref{EQ} and \eqref{EQ2}, we see that $$\Bbb E[Q_N(t)] \geq 4^{N+1} \alpha^2 \big( u(t-4^{-N-1})-u(t-2\cdot 4^{-N-1})\big).$$
By the Lebesgue differentiation theorem for nicely shrinking sets (see Theorem 3.21 in \cite{Fol}), there exists a measure zero zet $S \subset [0,\infty)$ such that for $t \notin S$, the right side of the last expression converges as $N \to \infty$ to $\alpha^2 u'(t) = \alpha^2 \int_0^1 (1-\mathbf v(t,z))dz.$ But we know that $\Bbb E[Q_N(t)] \to 0$ for every $t$, so we have shown that $u'(t)=0$ for all $t \notin S$. Thus $u(t) = u(0) = 0$. Since $\mathbf v \le 1$, this implies that $\mathbf v(s,z) = 1$ for a.e. $(s,z) \in [0,\infty) \times [0,1].$ Of course, there is nothing special about the interval $[0,1]$ here. By changing the definition of $Q_N(t)$ so that the sum ranges over all $k$ from $\lfloor 2^N a\rfloor$ to  $\lfloor 2^N b\rfloor$, we can obtain the same result on $[0,\infty) \times [a,b] $ for any real numbers $a<b$. 
\\
\\
We conclude that $\mathbf v = 1$ a.e. Thus by \eqref{v} we see that $\Bbb E[(\xi^1-\xi^2,\phi)^2] = 0$ for all $\phi\in L^2(\Bbb R_+\times \Bbb R)$, and thus $\xi^1=\xi^2$.
\end{proof}


Recall that a \textit{cylindrical Wiener process} is a family of Brownian motions $W_T(f)$, indexed by $f \in L^2(\Bbb R)$, defined on the same probability space, and satisfying $$\Bbb E[W_T(f) W_S(g)] = (S\wedge T)\langle f,g\rangle_{L^2(\Bbb R)}.$$ for all $f,g \in L^2$. Any space-time white noise $\xi$ defines a cylindrical Wiener process $W$, and vice versa, so the two may be viewed as equivalent objects \cite{Hai09}.
\begin{cor}\label{cor}
Suppose we have two standard cylindrical Wiener processes $W^1,W^2$ coupled onto the same probability space. Suppose furthermore that they satisfy the following conditions:
\begin{enumerate}
    
    \item $\Bbb E[W^1_T(f) W^2_T(g)] = 0$ for all $T \ge 0$ and all $f,g\in L^2(\Bbb R)$ which have disjoint supports.
    
    \item For $f,g \in L^2(\Bbb R)$, the processes $(W^1_T(f))_{T \ge 0}$ and $(W_T^2(g))_{T \ge 0}$ are both martingales with respect to their \textbf{joint} filtration.
    
    \item For every $t> 0$ the spatial process $h^2(t,\cdot) - h^1(t,\cdot)$ has a.s. finite $p$-variation for some $p<2$, where $h^i$ is a solution of $\partial_t h^i = \partial_x^2 h^i + F^i(\partial_xh^i)+ dW^i.$ Here $F^1,F^2$ are admissible nonlinearities as mentioned above.
\end{enumerate}
Then $W^1 = W^2$.
\end{cor}

\begin{proof}Define $\xi^1,\xi^2$ to be the random elements of $\mathcal S'(\Bbb R_+\times \Bbb R)$ such that $$(\xi^i, \phi) : = \int_0^\infty \langle \phi(T,\cdot), dW_T^i \rangle_{L^2(\Bbb R)}; \text{ for all } \phi \in \mathcal S(\Bbb R_+\times \Bbb R).$$ 
Note that $\Bbb E[(\xi^i,\phi)^2] = \|\phi\|_{L^2(\Bbb R_+\times \Bbb R)}^2$ so the $\xi^i$ are space-time white noises and we can stochastically extend the definition of $(\xi^i,\phi)$ to all $\phi \in L^2(\Bbb R_+\times \Bbb R)$.
\\
\\
Let $f,g \in L^2(\Bbb R)$. Since $W^1(f)$ and $W^2(g)$ are martingales in their \textit{joint} filtration we see that
$$\Bbb E\big[\big(W_T^1(f)-W_S^1(f)\big)\big(W_T^2(g)-W_S^2(g)\big)\big] = \Bbb E[ W^1_{T}(f) W^2_{T}(g)] - \Bbb E[ W^1_{S}(f) W^2_{S}(g)],$$ which equals zero whenever $f,g$ have disjoint supports. From this it follows (using approximation by elementary integrands) that $\Bbb E[(\xi^1,\phi)(\xi^2,\psi)]=0$ whenever $\phi$ and $\psi$ have disjoint supports on $\Bbb R_+\times \Bbb R$. Thus the conditions of Theorem \ref{mon1} are satisfied, so $\xi^1=\xi^2$, i.e., $W^1=W^2$.
\end{proof}

\section{Proof of the main theorem}

We will now derive some consequences of Theorem \ref{mon1} in the context of interacting particle systems. In particular we will prove Theorem \ref{mmrr} in the case of \cite{BG97} and \cite{CST18}. Although our results are for WASEP, they can be extended quite easily to some other systems, so we  describe in some generality a class of particle systems that we use.

\subsection{The basic coupling, height functions, and notation}

Although we consider ASEP for most of the paper, we would like to describe some extensions to more complicated models in later subsections. Thus we give a slightly more general description of the types of processes that are covered by our result.
\\
\\
In order to describe our result in full generality, fix $J \in \Bbb N$ and consider a function $b:\{-1,1\}\times \{0,...,J\}^2 \to [0,1]$. We consider Feller processes on the state space $S:= \{0,...,J\}^{\Bbb Z}$ which are described by the following dynamics. Each ordered pair $(x,x+1)$ and $(x,x-1)$ has a Poisson clock of rate 1. Every time the clock associated to $(x,y)$ rings, one particle jumps from $x$ to $y$ with probability $b(y-x,\eta(x),\eta(y))$ and stays there with probability $1-b(y-x,\eta(x),\eta(y)).$ However, the jump is suppressed if there is no particle at $x$, or if there are already $J$ particles at $y$ (equivalently we can just impose that $b(i,0,\cdot) = 0 = b(i,\cdot,J)$ for all $i=-1,1$). The pre-generator of such a process acts on local functions $f$ by the formula
\begin{equation}\label{gen}
    Lf(\eta) = \sum_{x,y\in\Bbb Z:|x-y|=1} b(y-x,\eta(x),\eta(y)) \big( f(\eta+e_y-e_x)-f(\eta)\big),
\end{equation}
where $e_x(z)=1_{\{x=z\}},$ and $f:S \to \Bbb R$ 
is some local function. This process is called a \textit{nearest-neighbor generalized-misanthrope process} if $b$ is increasing in the $\eta(x)$ variable and decreasing in the $\eta(y)$ variable. Examples include ASEP and more generally ASEP$(q,j)$ as considered in \cite{CST18}. See Subsection 3.5 for more on the latter.
\\
\\
For nearest-neighbor generalized misanthrope processes there is a natural way to run the dynamics associated to several initial data coupled together. This is usually called the \textbf{basic coupling. }Specifically for $x\in \Bbb Z$ we associate to each directed bond $(x,x+1)$ and $(x,x+1)$ Poisson clocks of rate one, as well as iid uniform random variables $\{U_i(x,x+1)\}_{i \ge 1}$ and $\{U_i(x+1,x)\}_{i \ge 1}$ which are independent of the Poisson clocks on that bond. Whenever the $i^{th}$ Poisson clock associated to $(x,x+1)$ rings, a particle jumps from $x$ to $x+1$ only when $b(1,\eta(x),\eta(x+1))<U_i(x,x+1),$ and similarly for $(x,x-1)$ with $b(-1,\eta(x),\eta(x-1))$ and $U_i(x,x-1).$ In this way, we can define a Markov process on the product $S\times S$ of the individual state spaces which describes the evolution of two particle systems coupled so that each marginal onto $S$ is a Feller process with generator $L$ given above, and moreover (by the monotonicity properties of $b$) the two individual particle systems stay dominated for all time if they start dominated (see \eqref{A} below). When $J=1$ there is a straightforward way to describe the coupling without any uniform variables, instead using Poisson clocks of different rates on each bond. For the seminal work on coupled processes, see e.g. \cite{Lig74, Har72}. Our description of the basic coupling is in the spirit of \cite{Har72}, while \cite{Lig74} instead chooses to explicitly write the generator for the entire coupled system on the product space.
\\
\\
If $(\eta_t(x))_{t \ge 0}$ is a generalized misanthrope process on the state space $\{0,...,J\}$ then we define the \textit{height function} $$h_t(x):=\begin{cases}h_t(0)+\sum_{k=0}^x (2\eta_t(k)-J),&\;\;\;\;\; x \ge 0,\\ h_t(0)+\sum_{k=0}^{-x} (2\eta_t(-k)-J),&\;\;\;\;\; x < 0,\end{cases},$$ where $h_t(0)$ equals twice the current through the origin up to time $t$, i.e., twice the number of particles which have moved from the site $x=0$ to the site $x=1$ minus twice the number of particles which have moved from the site $x=1$ to the site $x=0$ up to time $t$.
\\
\\
The height functions associated to nearest-neighbor misanthrope processes have two useful properties. The dynamics preserve their ordering as well as the ordering of their spatial derivative:
\begin{align}
    &h_t^1(x) \leq h_t^2(x) \text{ for all } t\ge 0,x \ge 0 \text{ if }h_0^1(x) \leq h_0^2(x) \text{ for all }x \ge 0\label{M}\tag{M}, \\ &
    \eta_t^1(x) \leq \eta_t^2(x) \text{ for all } t\ge 0,x \ge 0 \text{ if }\eta_0^1(x) \leq \eta_0^2(x) \text{ for all }x \ge 0\label{A}\tag{A}.
\end{align}

Property \eqref{M} is usually called \textit{monotonicity} of the particle system, whereas property \eqref{A} is usually called \textit{attractivity} of the system. Both properties are easily proved by considering the action of a single jump excitation in the joint system. In terms of the SPDE limits, \eqref{M} says that the limiting height functions $h^1$ and $h^2$ are coupled so that $h^1 \leq h^2$ if $h^1(0, \cdot) \leq h^2(0,\cdot)$, and \eqref{A} says that $h^2(t,\cdot) - h^1(t,\cdot)$ is a nondecreasing function for every $t > 0$ if it is nondecreasing for $t=0$.
\\
\\
Let us now establish some notation. A function $h: \Bbb Z \to \Bbb Z$ is called \textbf{viable} if there is a particle system associated to it, in other words if $h_t(x+1)-h_t(x) \in \{-J,-J+2,...,J-2,J\}$ for all $x$. Likewise a function from $\Bbb R\to \Bbb R$ will be called viable if its restriction to $\Bbb Z$ is viable and if its value at non-integers is linearly interpolated from the two nearest integer values. An obvious but important property used below is that the class of admissible height functions is closed under the operations $\max$ and $\min$.
\\
\\
Given some collection $h^1,...,h^n:\Bbb R_+\times \Bbb Z\to \Bbb R$ of time-evolving height profiles, we will often define ``rescaled and renormalized" versions of them which converge in law to the solution of the KPZ equation. In all of these cases what we will mean is that there exist some constants $a_{\epsilon}, b_{\epsilon}$ such that \begin{equation}h^{i,\epsilon}(t,x): = a_{\epsilon} h^{i} (\epsilon^{-2}t,\epsilon^{-1}x) +b_{\epsilon}t \label{resc}\end{equation} converges in law to the solution of \eqref{KPZ}.
\\
\\
Whenever we have an evolving height function $h(t,x)$ in our model, we will denote by $h^{\epsilon}$ its rescaled and renormalized version converging to KPZ. Thus $h^{\epsilon}$ is a random function from $\Bbb R_+\times \epsilon \Bbb Z \to \Bbb R$ that depends on $\epsilon$ in three different ways: through the initial data which is generally changing with $\epsilon$, through the parameters of the model which are being weakly scaled as $p=\frac12+\frac12\sqrt\epsilon$ and $q=\frac12-\frac12\sqrt\epsilon$ (this will be explained below), and through the renormalization constants and diffusive scaling as in \eqref{resc}. Often we will have several height functions $h^1,h^2,...,h^n$ which are coupled via the same dynamics, we will denote their rescaled versions as $h^{1,\epsilon},h^{2,\epsilon},...,h^{n,\epsilon}$. We will use the capital letters $(H^1,H^2,...,H^n)$ to denote the \textit{joint} continuum limits of the rescaled fields $(h^{1,\epsilon},h^{2,\epsilon},...,h^{n,\epsilon}).$ Thus the $H^i$ are random continuous functions from $\Bbb R_+\times \Bbb R \to \Bbb R$ which are defined on the same probability space as each other. We will always use the subscript $0$ to denote the initial data both in the prelimit and the limit, i.e., $H^i_0 = H^i(0,\cdot), h^{i,\epsilon}_0=h^{i,\epsilon}(0,\cdot),$ and so on.
\\
\\
Often we will have some initial data $h_0^{1,\epsilon},h_0^{2,\epsilon},..., h_0^{k,\epsilon}$ and from these we will build more initial data $h_0^{k+1,\epsilon},...,h_0^{k+n,\epsilon}$. We will always denote by $h^{i,\epsilon}$ (i.e., without the zero subscript) to denote the evolution of the \textit{coupled} the process started from $h_0^{i,\epsilon}$. In other words, the dynamics of the newly constructed $h^{i,\epsilon}$ are always implicitly assumed to be driven by the same realization of the Poissonian clocks (and uniform variables, if $J>1$) as those of the original $h^{i,\epsilon}$.
\\
\\
Whenever we refer to ``convergence" of $(h^{i,\epsilon}_0)_{i=1}^k$ to $(H_0^i)_{i=1}^k$, we mean convergence in $C(\Bbb R)^k$, where $C(\Bbb R)$ is the space of continuous functions on $\Bbb R$ equipped with the the topology of uniform convergence on compacts, which is completely metrizable via the same metric $$d(f,g):= \sum_{n\ge 1} 2^{-n} \max\big\{1,\sup_{x\in[-n,n]}|f(x)-g(x)|\big\}.$$ Sometimes we use the stronger topology of $\mathscr C_\delta^\alpha(\Bbb R)$ from Definition \ref{hold} and we specify whenever we do this. When we refer to the convergence of the entire height profile $h^{i,\epsilon}$ to $H^i$, we mean in the Skorohod space $D([0,T], C(\Bbb R)^k)$ for every $T>0$.

\subsection{The convergence result of Bertini-Giaomin}

Throughout Subsections 3.2, 3.2, and 3.4 we consider ASEP, which corresponds in \eqref{gen} to the choices $J=1$, $b(1,1,0) = p$, and $b(-1,1,0)=q$ where $p,q\ge 0$. In our $\epsilon$-dependent model below, $p$ will be scaled as $\frac12+\frac12\epsilon^{1/2}$ while $q$ will be scales as $\frac12-\frac12\epsilon^{1/2}$.
\\
\\
The main result of \cite{BG97} can be formulated as follows. We would like to emphasize once again that the height functions considered in the theorem below depend on $\epsilon$ in three different ways: through the initial data which is generally changing with $\epsilon$, through the parameters of the model which are being weakly scaled as $p=\frac12+\frac12\epsilon^{1/2}$ and $q=\frac12-\frac12\epsilon^{1/2}$, and through the renormalization constants and diffusive scaling as in \eqref{resc}.

\begin{thm}[Theorem 2.3 of \cite{BG97}]\label{bg97} Let $h_0^\epsilon$ be a deterministic sequence of initial data such that $\epsilon^{1/2} h_0^\epsilon(\epsilon^{-1}x)$ converges in $\mathscr C_\delta^\alpha$ to some $H_0$, where $0<\alpha<1/2$ and $0<\delta<1$. Let $h^{\epsilon}$ denote the rescaled and renormalized height function as in \eqref{resc}, with $a_\epsilon = \epsilon^{1/2}$ and $b_\epsilon = \frac12\epsilon^{-1}+\frac1{24}.$ Then $h^{\epsilon}$ converges in law to the Hopf-Cole solution of \eqref{KPZ}. The initial data of the limiting object is given by the limit in $\mathscr C_\delta^\alpha$ of $\epsilon^{1/2} h_0^\epsilon(\epsilon^{-1}x)$. The convergence is obtained with respect to the topology of the Skorohod space $D([0,T], C(\Bbb R))$, for all $T>0$.
\end{thm}

Let us remark that convergence in $\mathscr C_\delta^\alpha$ is slightly different than the actual assumption on the initial data given in \cite{BG97}. Specifically, in Definition 2.2 of \cite{BG97}, the authors considered \textit{possibly random} initial data which are ``near stationarity" in the sense that if $Z_0^\epsilon:=\exp(h_0^\epsilon)$ then one has the moment bounds $\|Z_0^\epsilon(x)\|_p\leq Ce^{ax}$ and $\|Z_0^\epsilon(x)-Z_0^\epsilon(y)\|_p\leq C|x-y|^{1/2} e^{a(|x|+|y|)},$ uniformly in $x,y,\epsilon$. Here $p$ is some exponent larger than 10 and $\|A\|_p:=\Bbb E[|A|^p]^{1/p}.$ The substance of their proof is unchanged when the exponent $1/2$ in the second bound is changed to arbitrary $\alpha\in(1/p,1/2)$. For technical reasons we will find it convenient to work with deterministic initial data which converge in $\mathscr C_\delta^\alpha$, which clearly satisfy these bounds. In fact even functions of \textit{linear} growth would satisfy these bounds, so our assumption of sublinear growth and deterministic data is actually substantially \textit{more} restrictive. We will randomize the assumptions on our initial data in Subsection 3.4.

\subsection{Main result: joint convergence for ASEP}

Our goal is to extend Theorem \ref{bg97} so that one may consider the limiting height field started from any finite collection of (sequences of) initial data $(h_0^{i,\epsilon})_{i=1}^k$ whose dynamics are jointly run according to the basic coupling. The goal is to obtain convergence in $D([0,T],C(\Bbb R)^k).$ We are going to do this in a manner which is essentially orthogonal to proof of the original convergence result of \cite{BG97}, by exploiting Theorem \ref{mon1} and \eqref{M} and \eqref{A}.

\begin{lem}\label{lem1} Suppose that we have two deterministic sequences of initial data $h_0^{1,\epsilon}$ and $h_0^{2,\epsilon}$ which both converge in $\mathscr C_\delta^\alpha$ to the same initial data. For any joint limit point $(H^1,H^2)$ of the basically coupled space-time processes, we have $H^1(t,x)=H^2(t,x)$ for all $t,x$ a.s.
\end{lem}

\begin{proof}
One readily checks that if two height functions are viable, then so are their maximum and minimum. We thus define $h_0^{3,\epsilon}: = \max\{h_0^{1,\epsilon},h_0^{2,\epsilon}\}$ and $h_0^{4,\epsilon}:=\min\{h_0^{1,\epsilon},h_0^{2,\epsilon}\}$. It is clear that $h_0^{3,\epsilon}$ and $h_0^{4,\epsilon}$ both converge in $\mathscr C_\delta^\alpha$ to the same initial data as $h_0^{1,\epsilon}$ and $h_0^{2,\epsilon}$. By \eqref{M} is also clear that $h^{4,\epsilon}(t,x) \leq h^{i,\epsilon}(t,x) \leq h^{3,\epsilon}(t,x)$ for $i=1,2$ and all $t,x,\epsilon$.
\\
\\
Letting $(H^1,H^2,H^3,H^4)$ denote a joint limit point of all four processes, we see that it must satisfy $H^4 \leq H^i \leq H^3$ for $i=1,2$. It is also true that $H^1_0=H^2_0=H^3_0=H^4_0$ because $h_0^{1,\epsilon}$ and $h_0^{2,\epsilon}$ converge in $\mathscr C_\delta^\alpha$ to the same function and hence so do their max and min. The KPZ equation satisfies uniqueness in law, thus two solutions started from the same initial data have the same expectation, i.e., $\Bbb E[H^4(t,x)] = \Bbb E[H^3(t,x)]$.
\\
\\
Since $H^4(t,x) \leq H^3(t,x)$ and $\Bbb E[H^4(t,x)] = \Bbb E[H^3(t,x)]$, we conclude that $H^4=H^3$ a.s. Since $H^1,H^2$ are nested in between $H^3$ and $H^4$, we conclude that $H^1=H^2$.
\end{proof}

The next lemma will be the key behind all subsequent results. It proves the main result in the very special case that the two initial data are ordered as in \eqref{A}, and it will be proved using the results of Section 2.

\begin{lem}\label{lem2}
If $h_0^{1,\epsilon}$ and $h_0^{2,\epsilon}$ which are both deterministic, their difference is nondecreasing for every $\epsilon$, and they converge weakly to initial data $H_0^1$ and $H_0^2$, then $h^{1,\epsilon}$ and $h^{2,\epsilon}$ converge jointly to the solution of the KPZ equation driven by the same noise.
\end{lem}

\begin{proof}
Note by \eqref{A} that the dynamic of the particle system preserves the condition that the difference of height functions is nondecreasing. Thus if $h_0^{2,\epsilon}-h_0^{1,\epsilon}$ is nondecreasing, then we know that $h^{2,\epsilon}(t,\cdot)-h^{1,\epsilon}(t,\cdot)$ is a.s. nondecreasing for every $t$. In particular, if $(H^1,H^2)$ is a joint limit point of $(h^{1,\epsilon},h^{2,\epsilon})$, then $H^2(t,\cdot) - H^1(t,\cdot)$ is nondecreasing (and in particular, of finite variation) for every $t$. Thus condition \textit{(3)} of Corollary \ref{cor} is satisfied. 
\\
\\
Now we just need to make sure that the conditions \textit{(1)} and \textit{(2)} of Corollary \ref{cor} is satisfied. For this we need to look into the precise details of how exactly Bertini and Giacomin proved their result. They first noted that of one defines $$Z^{i,\epsilon}_t(x):= \exp\big(\epsilon^{1/2}h^{i,\epsilon}(t,x) - (\frac12\epsilon^{-1} -\frac1{24})t\big),$$ then the $Z^{i,\epsilon}$ solve a discrete parabolic martingale-driven SPDE: 
\begin{equation}\label{discshe}
dZ_t(x) = (1+2\epsilon^{1/2})^{1/2}\Delta Z^{i,\epsilon}(x)dt+dM^{i,\epsilon}_t(x),
\end{equation}
where $\Delta f(x):= \frac12( f(x+1)+f(x-1)-2f(x))$, and $M^{i,\epsilon}(x)$ are jump martingales with the property that 
\begin{equation}\label{qv}\langle M^{i,\epsilon}(x),M^{j,\epsilon}(y)\rangle_t =0\;\;\;\;\;\;\;\;\text{if}\;\;\;\;\;\;\;\;x \neq y
\end{equation}
for all $i,j=1,2$. See equation (3.13) in \cite{BG97}.
\\
\\
Bertini and Giacomin then proceed to show that, for smooth functions $\phi \in C_c^{\infty}(\Bbb R)$, if one defines the martingale $M^{i,\epsilon}_t(\phi):= \epsilon\sum_{x\in \Bbb Z} \phi(\epsilon x) M_t^{i,\epsilon}(x)$, then any limit point (joint over all $\phi\in C_c^\infty$ and all $i=1,2$) of $M^{i,\epsilon}_t(\phi)$ as $\epsilon \to 0$ is a continuous martingale $M_t^i(\phi)$. In the language of \cite{Wal86}, the collection of martingales $M_t^i(\phi)$, as $\phi$ ranges over all smooth functions, form an \textit{orthogonal martingale measure}, in the sense that $\langle M_t^i(\phi),M_t^j(\psi)\rangle = 0$ whenever $\phi,\psi$ have disjoint supports and $i,j=1,2$ (this is clear because the corresponding statement is true even in the prelimit, by the property that $\langle M^{i,\epsilon}(x),M^{j,\epsilon}(y)\rangle_t =0$ if $x \neq y$ and $i,j=1,2$).
\\
\\
Now consider any joint limit point $(H^1,H^2)$ of $(h^{1,\epsilon},h^{2,\epsilon})$. Let $Z^1:=e^{H^1}$ and $Z^2:=e^{H^2}$.  Bertini and Giacomin show using \eqref{discshe} that the $Z^i$ must satisfy the relation $$(Z^i_t,\phi)_{L^2(\Bbb R)} - \int_0^t (Z^i_s,\phi'')_{L^2(\Bbb R)}ds = M^i_t(\phi).$$Using the language of \cite{KS88} and \cite{Wal86}, Bertini and Giacomin then use this to show that for each $i=1,2$ one can construct the driving noise $W^i$ of $Z^i$ as an Ito-Walsh stochastic integral against $M^i$. By the properties of Ito-Walsh stochastic integrals, it then automatically follows that $\langle W^i(\phi),W^j(\psi)\rangle=0$ for $i,j=1,2$ and $\phi,\psi$ of disjoint supports. Indeed, thus is is because the corresponding property is true for $M^i$ and because $W^i$ is a stochastic integral against $M^i$. Thus the conditions of Corollary \ref{cor} are satisfied and so $H^1=H^2$.
\end{proof}

\begin{rk} \label{>2}Note that the results of the previous two lemmas generalize fairly straightforwardly to the case where we have $k>2$ distinct initial data converging in $\mathscr C_\delta^\alpha(\Bbb R)^k.$ Indeed, if $(H^1,...,H^k)$ is a joint limit point of the height functions and if any subpair  $(H^i,H^j)$ is driven by the same noise, then they are all driven by the same noise. Here we are implicitly using the fact that the driving noise can be deterministically recovered from any solution-in-law of the KPZ equation, which is a nontrivial fact that can be deduced by combining the orthomartingale theory of \cite{Wal86} with the Hopf-Cole transform and a positivity result of \cite{Mue91}. Alternatively this fact can also be deduced more directly from the pathwise theories developed by \cite{Hai13, PR19}. Alternatively, even without using any of those aforementioned results, one can recognize that our proof strategy in both lemmas was done in such a way that the proofs generalize directly to several initial data. Indeed, in the proof of Lemma \ref{lem1} one can consider the max and min of $k$ distinct initial profiles, and these are still viable and convergent to the same limit in $\mathscr C_\delta^\alpha$. In the proof of Lemma \ref{lem2}, it is clear that one can keep track of both the noises as well as the height functions in the limit. Likewise, subsequent results such as Proposition \ref{prop}, Theorem \ref{mr}, and Theorem \ref{mr2} also generalize to more than two initial profiles, either by using the nontrivial fact mentioned earlier or by working through the logic in the proofs directly.
\end{rk}

Now that we have proved the main result in the special case when the two initial data are ordered deterministically at the level of the particle system, the next step will be to prove the claim for two initial data which are smooth or at least differentiable in some strong enough sense. Then we can dominate both of the initial data by some third initial data whose derivative is larger than both of the individual initial data, and then apply Lemma \ref{lem2} to conclude that the noises for all three are the same. This will be done in Proposition \ref{prop} below, but first we need a lemma.
\\
\\
We henceforth define $\mathcal V_\epsilon$ to be all functions of the form $\epsilon^{1/2}f(\epsilon^{-1}x)$ where $f$ is a viable height function as defined in Subsection 3.1. We also let $\mathscr C^1_\delta$ to be the set of all continuously differentiable functions on $\Bbb R$ such that 
$\sup_{x\in\Bbb R} (1+|x|)^{-\delta}(|f'(x)|+\int_0^x |f'(u)|du)<\infty.$ It is clear that $\mathscr C^1_\delta$ is a Banach space if we define its norm to be $|f(0)|$ plus that quantity\footnote{This is \textbf{not} a standard definition of $\mathscr C^1_\delta$, we have only defined it in this way for convenience of the arguments given later. Strictly speaking, we should really call this space $BV_\delta$ or something similar due to the defining condition that the variation is bounded by $C|x|^\delta$. \iffalse In general Hölder spaces do not generalize unambiguously to integer exponents, as they could be used to denote e.g. Lipchitz functions, continuously differentiable functions, Zygmund class functions, functions with derivatives in BMO, and other types of function spaces.\fi} and that it embeds compactly into $\mathscr C^\alpha_{\delta'}$ whenever $\alpha<1$ and $\delta'>\delta.$ More generally we will often use the fact that if $0<\alpha_1<\alpha_2\leq 1$ and $1>\delta_1>\delta_2>0$ then $\mathscr C^{\alpha_2}_{\delta_2}$ embeds compactly into $\mathscr C^{\alpha_1}_{\delta_1}.$ This follows from Arzela-Ascoli together with the interpolation properties of Hölder seminorms, see e.g. Lemma 24.14 of \cite{Dri} for the elementary proof, or \cite{Mey90} for a more general theory on Hölder spaces and their embeddings via Littlewood-Paley theory (Section 2 of \cite{PR19} also has a nice discussion of the latter). We now give an approximation algorithm $\mathcal A^\epsilon$ for smooth functions by rescaled viable functions, and moreover the algorithm preserves the property that the difference of two functions is nondecreasing.

\begin{lem}\label{two}
Fix $\alpha\in(0,1/4)$ and $\delta<\delta' \in (0,1)$. Then there exists a family of maps $\mathcal A^\epsilon : \mathscr C^1_\delta \to \mathscr C^\alpha_\delta \cap \mathcal V_\epsilon$ with the following properties:
\begin{itemize}
    \item For all $f\in \mathscr C^1_\delta$, we have that $\mathcal A^\epsilon(f) \to f$ in $\mathscr C^{\alpha}_{\delta'}$ as $\epsilon \to 0$.
    
    \item $\mathcal A^\epsilon(g)-\mathcal A^\epsilon(f)$ is nondecreasing whenever $g-f$ is nondecreasing.
\end{itemize}
\end{lem}

\begin{proof}
We will construct $\mathcal A^\epsilon(f)$ on $\epsilon\Bbb Z$. The values in between are understood to be linearly interpolated.
\\
\\
Note that $(1+|x|)^{-\delta'}(|f(x)|+|f'(x)|)$ can be viewed as a continuous function on the \textit{closed} interval $[-\infty,\infty],$ which vanishes at the endpoints $-\infty$ and $\infty$. Suppose $(1+|x|)^{-\delta'}(|f(x)|+|f'(x)|)<2\epsilon^{1/2}$ on $[-\infty, -M]\cup [M,\infty]$ where implicitly $M=M_\epsilon\geq 1$. We define $\mathcal A^\epsilon(f)$ on $(-\infty, -M]$ to just oscillate between the two values in $\epsilon^{1/2} \Bbb Z$ which are closest to $f(-M)$.
\\
\\
Next we define $\mathcal A^\epsilon(f)$ on the interval $[-M,M]$. Break $[-M,M]$ into $\lfloor \epsilon^{-1/4} \rfloor$ equally sized intervals of length $2M/\lfloor\epsilon^{-1/4}\rfloor.$ On each of those intervals $I$, let $x^\epsilon_I:=\min I\cap \epsilon \Bbb Z.$ For $x\in \epsilon \Bbb Z$ such that $x_I\leq x\leq \max\{ x_I + |f'(x_I)|\epsilon^{3/4}, \max I\cap \epsilon \Bbb Z\}$ we define $f$ inductively by the formula $\mathcal A^\epsilon f(x+\epsilon)-\mathcal A^\epsilon f(x):=$ sign$(f'(x_I))\epsilon^{1/2}$. For $x \in \epsilon \Bbb Z$ such that $x_I+|f'(x_I)|\epsilon^{3/4} \leq x \leq \max I\cap \epsilon\Bbb Z$ we simply define $\mathcal A^\epsilon f(x+\epsilon)-\mathcal A^\epsilon f(x):=\epsilon^{1/2} (-1)^{x/\epsilon}.$
\\
\\
Finally, define $\mathcal A^\epsilon(f)$ on $[M,\infty)$ by the formula $\mathcal A^\epsilon(f)(x+\epsilon)-\mathcal A^\epsilon (f)(x):= \epsilon^{1/2} (-1)^{x/\epsilon}.$ 
\\
\\
From our construction it is clear that $\mathcal A^\epsilon(g)-\mathcal A^\epsilon(f)$ is nondecreasing whenever $g-f$ is nondecreasing. This is because the latter is equivalent to $g'\geq f'.$
\\
\\
Note that for all $x \in [-M,M]$, the quantity $\mathcal A^\epsilon (f)(x)$ is always within $O(\epsilon^{1/4})$ of $f(-M)+\epsilon^{1/4} \sum_{I:\sup I<x} f'(x_I),$ which is a Riemann sum approximation to $f(-M) + \int_{-M}^x f'(t)dt$. Consequently $\mathcal A^\epsilon (f)$ converges pointwise to $f$ as $\epsilon \to 0$.
\\
\\
Next we prove that there exists $C>0$ independent of $x,y, \epsilon$ (but in general dependent on $f$) such that $|\mathcal A^\epsilon f(x)-\mathcal A^\epsilon f(y)|\leq C|x|^\delta |x-y|^{1/4}$ whenever $|y-x|\leq 1$ and $x \in \Bbb R.$ This is enough to prove relative precompactness of $\{\mathcal A^\epsilon (f)\}_{\epsilon \in (0,1]}$ inside of $\mathscr C^\alpha_{\delta'}$ (because $\alpha<1/2$ and $\delta'>\delta$), which would finish the proof. To prove this inequality, we first consider the case where $x,y\in [-M,M]$ and $1\ge |x-y|>\epsilon^{3/4}.$ In this case, note that $\mathcal A^\epsilon f(x)-\mathcal A^\epsilon f(y)$ is always within $\|f'\|_{L^{\infty}([y,x])} \epsilon^{1/4}$ of $\epsilon^{1/4} \sum_{I: y<\sup I < x} f'(x_I)$. Now the number of intervals $I$ in the approximation scheme such that such that $y<\sup I<x$ is bounded above by $\epsilon^{-1/4}/2M \leq \epsilon^{-1/4}.$ Consequently we find that $|\mathcal A^\epsilon f(x)-\mathcal A^\epsilon f(y)| \leq 2\epsilon^{1/4} \|f'\|_{L^\infty([y,x])}.$ Now since $f \in \mathscr C^1_\delta$ and $|y-x| \leq 1$ we find that $\|f'\|_{L^\infty[y,x]} \leq \|f\|_{\mathscr C^1_\delta} |x|^{\delta},$ proving the claim in this case since $\epsilon^{1/4}$ can be bounded above by $|x-y|^{1/3}$ (recall we assumed $|x-y|>\epsilon^{3/4})$. Next we consider the case where $\epsilon<|x-y|\leq \epsilon^{3/4}. $ Then we can use the naive bound 
\begin{align*}|\mathcal A^\epsilon f(x)-\mathcal A^\epsilon f(y)| &\leq \epsilon^{1/2} + \sum_{u \in \epsilon \Bbb Z \cap [x,y]} |\mathcal A^\epsilon f(u+\epsilon)-\mathcal A^\epsilon f(u)| \\& \leq \epsilon^{1/2} + \epsilon^{-1/2}|x-y| \\&\leq \epsilon^{1/2} + \epsilon^{1/4} \leq 2|x-y|^{1/4}.
\end{align*}
Finally, we consider the case where $|x-y|<\epsilon.$ In this case it is clear that since the global Lipchitz constant of $\mathcal A^\epsilon (f)$ never exceeds $\epsilon^{-1/2}$ that one has $$|\mathcal A^\epsilon(f)(x)-\mathcal A^\epsilon(f)(y)| \leq \epsilon^{-1/2}|x-y|\leq \epsilon^{-1/2} \epsilon^{1/2} |x-y|^{1/2} = |x-y|^{1/2}.$$
\end{proof}

\begin{prop}\label{prop} Suppose that we have two viable deterministic sequences of initial data such that their re-scaled versions $h_0^{1,\epsilon}$ and $h_0^{2,\epsilon}$ converge in $\mathscr C^\alpha_\delta$ to functions $H_0^1$ and $H_0^2$ respectively, where $0<\alpha<1/2$ and $0<\delta<1.$ Assume that $H_0^1,H_0^2\in \mathscr C^1_\delta$. Then for any joint limit point $(H^1,H^2)$ of $(h^{1,\epsilon},h^{2,\epsilon})$, $H^1$ and $H^2$ are solutions of the KPZ equation driven by the same noise.
\end{prop}

\begin{proof}
Choose some probability space $(\Omega, \mathcal F, \Bbb P)$ on which one may define, for each $\epsilon\in (0,1],$ a system of i.i.d. Poisson clocks of rate $p=\frac12+\frac12\sqrt\epsilon$ and rate $q=\frac12-\frac12\sqrt{\epsilon}$ associated to each bond $\{x,x+1\}$ with $x \in \Bbb Z$. For different values of $\epsilon$ these can be coupled in an arbitrary manner; ultimately it is irrelevant.
\\
\\
Define $H^3_0(x): = \int_0^x \max\{\partial_xH^1_0(u),\partial_xH^2_0(u)\} du,$ so that $H^3_0-H^1_0$ and $H^3_0 - H^2_0$ are both nondecreasing functions. Note that $H_0^3$ also lies in $\mathscr C_\delta^1.$
\\
\\
Next, use the algorithm in Lemma \ref{two} to construct viable height functions $h_0^{i,\epsilon}$ for $3\le i \le 5$ in such a way that
\begin{itemize}
    \item $\epsilon^{1/2}h^{3,\epsilon}_0(\epsilon^{-1}x)$ converges in $\mathscr C_{\delta'}^{\alpha'}$ to $H^3(x)$ for some $\delta'>\delta$ and $\alpha'<\alpha$.
    
    \item $\epsilon^{1/2}h^{4,\epsilon}_0(\epsilon^{-1}x)$ converges in $\mathscr C_{\delta'}^{\alpha'}$ to $H^1(x)$ for some $\delta'>\delta$ and $\alpha'<\alpha$.
    
    \item $\epsilon^{1/2}h^{5,\epsilon}_0(\epsilon^{-1}x)$ converges in $\mathscr C_{\delta'}^{\alpha'}$ to $H^2(x)$ for some $\delta'>\delta$ and $\alpha'<\alpha$.
    
    
    \item $h_0^{3,\epsilon}-h_0^{i,\epsilon}$ are nondecreasing in $x$ for $i=4,5$ and all $\epsilon$.

\end{itemize}

On the same probability space, let $h^{1,\epsilon}$ and $h^{2,\epsilon}$ be the (time-evolving) height profiles started from initial data $h_0^{1,\epsilon}$ and $h_0^{2,\epsilon}$ (respectively) and whose dynamics are governed by the Poisson clocks described above. Then let $h^{3,\epsilon}, h^{4,\epsilon}$ and $h^{5,\epsilon}$ be the height functions associated with initial data $h_0^{3,\epsilon},h_0^{4,\epsilon}, h_0^{5,\epsilon}$, respectively.
\\
\\
Let $(H^1,H^2,H^3,H^4,H^5)$ be a joint limit point of $(h^{1,\epsilon},h^{2,\epsilon},h^{3,\epsilon},h^{4,\epsilon},h^{5,\epsilon})$, which is \textbf{not} necessarily defined on the same probability space as above. Then let $(\xi^1,\xi^2,\xi^3,\xi^4,\xi^5)$ denote the respective driving noises. By Lemma \ref{lem1} we know that $H^4=H^1$ and $H^5=H^2$, therefore $\xi^4=\xi^1$ and $\xi^5=\xi^2$ (for instance by Theorem \ref{mon1}). But since $h_0^{3,\epsilon}-h_0^{i,\epsilon}$ are nondecreasing in $x$ for $i\in\{4,5\}$, Lemma \ref{lem2} implies that $\xi^4=\xi^3$ and $\xi^5=\xi^3$. So all noises are equal, and in particular $\xi^1=\xi^2$.
\end{proof}

We are now ready to state and prove the main result for two arbitrary initial data which converge in $\mathscr C_\delta^\alpha$. The idea will be to nest the two initial data between smooth initial data satisfying the hypotheses of Proposition \ref{prop}, and then take advantage of the monotonicity \eqref{M}. Recall our notation that the subscript ``0" in $H_0^i$ denotes the (deterministic) initial data of the height profile, while $H^i$ without the subscript denotes the entire space-time profile viewed as a random variable in some Skorohod space $D([0,\infty),C(\Bbb R))$.

\begin{thm}\label{mr}
Let $\eta_t^{1,\epsilon}$ and $\eta_t^{2,\epsilon}$ be sequences (indexed by $\epsilon\in(0,1]$) of exclusion processes with generator \eqref{gen} on $\{0,1\}^{\Bbb Z}$ with $b(-1,1,0) = \frac12+\frac12\sqrt \epsilon$ and $b(1,1,0)=\frac12-\frac12\sqrt\epsilon$. Assume that the dynamics are run via the basic coupling as described above in Subsection 3.1. Let $h^{1,\epsilon},h^{2,\epsilon}$ denote the rescaled height functions as in \eqref{resc}. Suppose that the deterministic sequences of initial data $h_0^{1,\epsilon},h_0^{2,\epsilon}$ converge in $\mathscr C_\delta^\alpha$ as $\epsilon \to 0$ to $H_0^1,H_0^2$ respectively, where $0<\alpha<1/2$ and $0<\delta<1$. Then one has joint convergence in law as $\epsilon \to 0$ of the entire time-evloving height profile $(h^{1,\epsilon},h^{2,\epsilon})$ to $(H^1,H^2)$ where $H^1,H^2$ both solve the KPZ equation with the same noise and with initial data $H_0^1,H_0^2,$ resp. The convergence holds with respect to the topology of $D([0,T],C(\Bbb R)^2).$
\end{thm}

\begin{proof}
Choose some probability space $(\Omega, \mathcal F, \Bbb P)$ on which one may define, for each $\epsilon\in(0,1],$ a system of i.i.d. Poisson clocks of rate $p=\frac12+\frac12\sqrt\epsilon$ and rate $q=\frac12-\frac12\sqrt{\epsilon}$ associated to each bond $\{x,x+1\}$ with $x \in \Bbb Z$. For different values of $\epsilon$ these can be coupled in an arbitrary manner; ultimately it is irrelevant.
\\
\\
Choose arbitrary sequences of smooth approximating height functions $a_0^{N},b_0^{N}, c_0^{N}, d_0^{N}$, for $N \in \Bbb N$, satisfying the following five properties:
\begin{itemize}
    \item $a_0^{N},b_0^{N}, c_0^{N}, d_0^{N} \in \mathscr C^1_{\delta'}$ for some $\delta'>\delta$.
    
    \item $b_0^N \leq H^1_0\leq a_0^N$.
    
    \item $d_0^N \leq H^2_0 \leq c_0^N.$
    
    \item $a_0^N,b_0^N$ converge in $\mathscr C^{\alpha'}_{\delta'}$ to $H^1_0$ for some $\alpha'<\alpha$ and $\delta'$ as above.
    
    \item $c_0^N,d_0^N$ converge in $\mathscr C^{\alpha'}_{\delta'}$ to $H^2_0$ with $\alpha',\delta'$ as above.

\end{itemize}
The existence of such sequences is straightforward. Indeed, one can even choose $(\alpha',\delta')$ arbitrarily from $(0,\alpha)\times (\delta,1)$ and then take $a_0^N$ and $c_0^N$ to coincide with $|x|^{(\delta+\delta')/2}$ in some neighborhood of $\pm\infty$ and similarly one can choose $b_0^N$ and $d_0^N$ to coincide with $-|x|^{(\delta+\delta')/2}$ in some neighborhood in $\pm\infty$ (this neighborhood will obviously depend on $N$ though).
\\
\\
Now use Lemma \ref{two} to define viable height functions $a_0^{N,\epsilon},b_0^{N,\epsilon},c_0^{N,\epsilon},d_0^{N,\epsilon}$ which are jointly admissible and converge under the appropriate scaling to $a_0^N,b_0^N,c_0^N,d_0^N$, respectively. Let $(h^{1,\epsilon}_t,h^{2,\epsilon}_t,a^{N,\epsilon}_t,b^{N,\epsilon}_t,c^{N,\epsilon}_t, d^{N,\epsilon}_t)_{t \ge 0}$ denote the (time-evolving) height profiles associated with initial data $(h^{1,\epsilon}_0,h^{2,\epsilon}_0,a^{N,\epsilon}_0,b^{N,\epsilon}_0,c^{N,\epsilon}_0, d^{N,\epsilon}_0),$ respectively. The dynamics for each of these objects are run according to the Poisson clocks described above.
\\
\\
Let $(H^1,H^2,a^N,b^N,c^N,d^N)$ denote a joint limit point of all of these objects (as $\epsilon \to 0)$, which is not necessarily defined on the same probability space. By Proposition \ref{prop}, all of $a^N,b^N,c^N,d^N$ solve the KPZ equation with the same realization of the noise $\xi$ (we are using Remark \ref{>2} here). 
\\
\\
Recall by construction, $a_0^N$ and $b_0^N$ converge in $\mathscr C_{\delta'}^{\alpha'}$ to $H_0^1$ as $N \to \infty$, and moreover $b^N_0 \leq H^1_0 \leq a^N_0$. Note that for a fixed realization of $\xi$, the solution of the KPZ equation is continuous as a function of the initial data, viewed as a function from $ \mathscr C^{\alpha'}_{\delta'} (\Bbb R)\to C(\Bbb R_+\times \Bbb R)$. This can be proved directly from Definition \ref{hopf} by exploiting Mueller's positivity result \cite{Mue91}, and it can also be proved more directly by using more modern techniques such as \cite{Hai13, PR19}. Thus, $b^N-a^N$ converges uniformly to $0$ on compact sets of $\Bbb R_+\times \Bbb R$, and moreover by \eqref{M} it is true that $b^N \leq H^1 \leq a^N$. Thus $a^N,b^N$ both converge uniformly as $N\to \infty$ to $H^1$ on compact subsets of $\Bbb R_+\times \Bbb R$, and on the other hand they also converge to the solution of the KPZ equation driven by the common noise of the $a^i,b^i$ and initial data $H^1_0$. Thus, we conclude that $H^1$ is driven by the same noise as the $a^i,b^i$. 
\\
\\
A completely analogous argument will show that $H^2$ is driven by the same noise as the $c^i,d^i$, completing the proof.
\end{proof}

One can ask why, in the above proof, one could not have defined simpler approximations $H^i * \phi_{\delta}$ and then just used the fact the the KPZ equation is continuous as a function of the initial data on $\mathscr C^{\alpha'}_{\delta'}(\Bbb R)$, and used Proposition \ref{prop} without relying on \eqref{M}. The problem with this idea is that it would be circular: we do not know beforehand that the joint limit points all solve the KPZ equation with the same noise: therefore we do not know that they are continuous as a function of the initial data. Hence some kind of monotonicity property must be leveraged. 

\subsection{Random initial conditions near stationarity}

In the above theorem we assumed that $h_0^{1,\epsilon}$ and $h_0^{2,\epsilon}$ were deterministic and converged with respect to the topology of $\mathscr C_\delta^\alpha(\Bbb R)$ for some $0<\alpha<1/2$ and some $0<\delta<1.$ In this subsection we relax these conditions slightly to allow for \textit{random} sequences of pairs of initial data that may only converge in distribution and satisfy some $p^{th}$ moment bounds that are generally easy to check in practice. The prototypical examples to keep in mind for this subsection are the height function pairs generated by iid Bernoulli configurations. These two product Bernoulli configurations may be independent or correlated by some parameter; it does not matter so long as the finite-dimensional marginals for the pair of height functions converge \textit{jointly} in law.
\\
\\
We denote by $\|X\|_p:=\Bbb E[|X|^p]^{1/p}$ for a random variable $X$ defined on some probability space.

\begin{thm}\label{mr2}
The conclusion of Theorem \ref{mr} still holds for random initial data $(h_0^{1,\epsilon},h_0^{2,\epsilon})$ so long as this pair converges jointly in the sense of finite dimensional distributions to $(H_0^1,H_0^2)$ and there exist $\alpha\in (0,1/2]$,$\delta \in(0,1), C>0$, and $p>\max\{\alpha^{-1},(1-\delta)^{-1}\}$ such that for all $\epsilon\in(0,1]$ the pair satisfies the moment bounds $$\| h_0^{i,\epsilon}(x)\|_p \leq C(1+|x|)^{\delta} \text{ for all } x\in \Bbb R$$
$$\|h_0^{i,\epsilon}(x)-h_0^{i,\epsilon}(y)\|_p \leq C(1+|x|)^\delta|x-y|^\alpha \text{ whenever } |x-y|\leq 1.$$
\end{thm}

The proof is immediately obtained by combining the results of Lemmas \ref{3.4.1} and \ref{3.4.2} given just below. Note that the rescaled height functions associated to iid Bernoulli configurations satisfy these bounds with $\alpha=\delta=1/2$ (in fact, one does not even need the extra factor of $(1+|x|)^\delta$ in the second bound).

\begin{lem}\label{3.4.1}
The conclusion of Theorem \ref{mr} still holds for random initial data $(h_0^{1,\epsilon},h_0^{2,\epsilon})$ so long as this pair converges in law to $(H_0^1,H_0^2)$ with respect to the topology of $\mathscr C_\delta^\alpha \times \mathscr C_\delta^\alpha$ for some $\alpha\in (0,1/2)$ and $\delta \in (0,1).$
\end{lem}

\begin{proof}
Recall $\mathcal V^\epsilon$ which is the set of all functions of the form $\epsilon^{1/2} h(\epsilon^{-1}x)$ where $h$ is a viable height function. Fix $T>0$ and let $\mathcal M$ denote the set of all probability measures on $D([0,T],C(\Bbb R)^2)$. Define $Q^\epsilon: \mathcal (V^\epsilon \cap \mathscr C_\delta^\alpha)^2 \to \mathcal M$ by sending a rescaled pair of viable height functions $(h^1_0,h^2_0)$ to the law of the (entire time evolution of the) basically coupled ASEP height process started from $(h^1_0,h^2_0)$, with right jump parameter $\frac12+\frac12\epsilon^{1/2}$ and left jump parameter $\frac12-\frac12\epsilon^{1/2}.$
\\
\\
Likewise, define $Q: \mathscr C_\delta^\alpha(\Bbb R)^2\to \mathcal M$ by sending $(h^1_0,h^2_0)$ the solution of the KPZ equation driven by the same realization of $\xi$ started from $h^1_0,h^2_0$ respectively. Theorem \ref{mr} says precisely that $Q^\epsilon(h_0^{1,\epsilon},h_0^{2,\epsilon}) \to Q(h_0^1,h_0^2)$ whenever $(h_0^{1,\epsilon},h_0^{2,\epsilon}) \to (h_0^1,h_0^{2})$ in $(\mathscr C_\delta^\alpha)^2.$
\\
\\
Now suppose that the hypothesis of the lemma holds, i.e., $(h_0^{1,\epsilon},h_0^{2,\epsilon})$ converges in law to $(H_0^1,H_0^2)$ with respect to the topology of $\mathscr C_\delta^\alpha(\Bbb R)^2$. By Skorohod's representation theorem\footnote{One technical remark here is that the spaces $\mathscr C_\delta^\alpha$ are not separable and thus Skorohod's representation theorem may not hold, strictly speaking. In practice this is not an issue, because for $\delta<\delta'$ and $\alpha'<\alpha$ it is actually true that $\mathscr C_\delta^\alpha$ embeds \textit{compactly} into $\mathscr C_{\delta'}^{\alpha'}$, as we already mentioned earlier. Any compact metric space is separable, thus in our argument above, one should instead use almost sure convergence with respect to the weaker topology of $\mathscr C^{\alpha'}_{\delta'}$ for some $\alpha'<\alpha$ and $\delta'>\delta$. This does not cause any issues for the proof.\iffalse into the separable closed subspace of $\mathscr C_{\delta'}^{\alpha'}$ given by the closure of compactly supported smooth functions in $\mathscr C_{\delta'}^{\alpha'}$, see e.g. Theorems 2.52 and 8.27 in \cite{Wea18}. Thus in our argument above, one should instead use almost sure convergence with respect to the weaker topology of $\mathscr C^{\alpha'}_{\delta'}$ for some $\alpha'<\alpha$ and $\delta'>\delta$. This does not cause any issues for the proof.\fi} we may find a probability space $(\Omega,\mathcal F,\Bbb P)$ such that $(h_0^{1,\epsilon},h_0^{2,\epsilon}) \to (h_0^{1},h_0^{2})$ in $(\mathscr C_\delta^\alpha)^2$ almost surely. Then by the discussion above, $Q^\epsilon(h_0^{1\epsilon},h_0^{2,\epsilon}) \to Q(h_0^{1},h_0^{2})$ in $\mathcal M$ almost surely. This is enough to give the required result. Indeed it shows that $\Bbb E[f\big(Q^\epsilon(h_0^{1\epsilon},h_0^{2,\epsilon})\big)] \to \Bbb E[f\big( Q(h_0^{1},h_0^{2})\big)]$ for all bounded continuous $f:\mathcal M\to \Bbb R.$ To finish the proof one simply takes $f$ of the form $f(\nu):=\int_{D([0,T],C(\Bbb R))} g(h) \nu(dh)$ where $g$ is a bounded real-valued continuous function on $D([0,T],C(\Bbb R)).$ Then one may disintegrate the law of $h^{i,\epsilon}$ by decoupling the initial data and the dynamics to obtain the desired result.
\end{proof}

\begin{lem}\label{3.4.2}
Suppose that $\{h^{\epsilon}\}_{\epsilon \in (0,1]}$ is a family of $C(\Bbb R)$-valued random variables such that there exist $\alpha\in (0,1/2)$,$\delta \in(0,1), C>0$, and $p>\max\{\alpha^{-1},(1-\delta)^{-1}\}$ which satisfy the following moment bounds uniformly over all $\epsilon\in (0,1]$:$$\| h^{\epsilon}(x)\|_p \leq C(1+|x|)^{\delta} \text{ for all } x\in \Bbb R$$
$$\|h^{\epsilon}(x)-h^{\epsilon}(y)\|_p \leq C(1+|x|)^\delta |x-y|^\alpha \text{ whenever } |x-y|\leq 1.$$
Then there exist $\alpha'\in (0,\alpha)$ and $\delta'\in(\delta,1)$ such that $\{h^\epsilon\}_{\epsilon \in (0,1]}$ is tight with respect to the topology of $\mathscr C_{\delta'}^{\alpha'}.$
\end{lem}

\begin{proof}
Recall from earlier that $\mathscr C_{\delta}^{\alpha}$ embeds \textit{compactly} into $\mathscr C_{\delta'}^{\alpha'}$ whenever $\delta'>\delta$ and $\alpha'<\alpha.$ Therefore to prove the lemma, it suffices to show that if the two inequalities in the lemma statement hold, then there exist $\alpha',\delta'$ such that $$\lim_{a\to \infty} \sup_{\epsilon \in (0,1]} \Bbb P(\|h^\epsilon\|_{ \mathscr C_{\delta'}^{\alpha'} }>a) = 0.$$
We actually show something stronger, namely that under the given assumptions, there exists $C'>0$ such that for all $a>0$ \begin{equation}\label{8}\sup_{\epsilon \in (0,1]} \Bbb P(\|h^\epsilon\|_{ \mathscr C^{\alpha'}_{\delta'}} > a) \leq C'a^{-p},
\end{equation}
where $p$ is the same exponent given in the lemma statement. 
To prove this we write $\|h^\epsilon\|_{ \mathscr C^{\alpha'}_{\delta'}} = \|h^\epsilon\|_{\delta'}+ [h^\epsilon]_{\alpha',\delta'}$ where $\|h\|_{\delta'}:= \sup_{x\in\Bbb R} \frac{|h(x)|}{(1+|x|)^{\delta'}}$ and $[h]_{\alpha',\delta'}:= \sup_{x\in \Bbb R} (1+|x|)^{-\delta'} \sup_{|y-x|\leq 1} \frac{|h(x)-h(y)|}{|x-y|^{\alpha'}}. $ 
\\
\\
To prove \eqref{8}, the following fact will be useful to us: For any $\gamma \in (0,1),$ the $\gamma$-Hölder seminorm $[f]_{\gamma}$ of a function $f: [0,1]\to \Bbb R$ is equivalent (as a seminorm) to the quantity given by $\sup_{n \in \Bbb N, 1 \le k \le 2^n} 2^{\gamma n} |f(k2^{-n})-f((k-1)2^{-n})|.$ This is proved as an intermediate step in the standard proof of the classical Kolmogorov-Chentsov criterion. 
\\
\\
The exact choices of $\alpha',\delta'$ will be specified later, but for now let them denote generic constants. Now to prove \eqref{8} let us write for a function $h$,
\begin{align*}
\|h\|_{\delta'} &\leq \sup_{n\in \Bbb Z} (1+|n|)^{-\delta'} \big(|h(n)| + \sup_{x\in [n,n+1]} |h(x)-h(n)|\big) \\ &\leq \sup_{n\in \Bbb Z} (1+|n|)^{-\delta'} \bigg(|h(n)| + \sup_{x\in [n,n+1]} \frac{|h(x)-h(n)|}{|x-n|^{\alpha'}}\bigg) \\ &\lesssim \sup_{n\in \Bbb Z} (1+|n|)^{-\delta'}\bigg(|h(n)| + \sup_{r \in \Bbb N, 1\leq k \leq 2^r} 2^{\alpha'r} |h(n+k2^{-r})-h(n+(k-1)2^{-r})| \bigg),
\end{align*}
where $\lesssim$ denotes the absorption of some universal constant which can depend on $\alpha',\delta'$ but not on the function $h$. Likewise let us note that
\begin{equation*}
    [h]_{\alpha',\delta'}\lesssim \sup_{n\in \Bbb Z} (1+|n|)^{-\delta} \sup_{r \in \Bbb N, 1\leq k \leq 2^r} 2^{\alpha'r} |h(n+k2^{-r})-h(n+(k-1)2^{-r})|.
\end{equation*}
Consequently we find that $$\|h\|_{\mathscr C_{\delta'}^{\alpha'} }\lesssim A(h,\delta')+B(h,\alpha',\delta'),$$
where 
\begin{align*}A(h,\delta')&:= \sup_{n\in \Bbb Z} (1+|n|)^{-\delta'} |h(n)|,
\\ B(h,\alpha',\delta')&:= \sup_{n\in \Bbb Z} (1+|n|)^{-\delta'} \sup_{r \in \Bbb N, 1\leq k \leq 2^r} 2^{\alpha'r} |h(n+k2^{-r})-h(n+(k-1)2^{-r})|.
\end{align*}
Now, with $h^\epsilon$ as given in the lemma statement, let us bound these terms $A(h^\epsilon,\delta')$ and $B(h^\epsilon, \alpha',\delta')$ individually to obtain \eqref{8}. We will do this by using the hypotheses in the lemma. Note that by a brutal union bound and Markov's inequality followed by the hypothesis $\|h^\epsilon(x)\|_p \leq (1+|x|)^{\delta},$ we have
\begin{align*}
    \Bbb P( A(h^\epsilon,\delta')>a) &\leq \sum_{n\in \Bbb Z} \Bbb P(|h^\epsilon(n)|>(1+|n|)^{\delta'} a) \\ &\leq \sum_{n \in \Bbb Z} a^{-p}(1+|n|)^{-\delta'p}E[|h^\epsilon(n)|^p] \\ &\leq a^{-p} \sum_{n \in \Bbb Z} (1+|n|)^{(\delta-\delta')p},
\end{align*}
The series converges as long as $\delta'$ is chosen so that $(\delta-\delta')p<-1$, for instance $\delta':= \frac12(1+\delta+\frac1p)$ which is less than $1$ by the hypothesis that $p>(1-\delta)^{-1}.$ Next we control $B$, which will also just use a brutal union bound and Markov's inequality:
\begin{align*}
    \Bbb P(B(h^\epsilon,\alpha',\delta')>a) &\leq \sum_{\substack{n\in\Bbb Z\\r\in \Bbb N\\1\leq k \leq 2^r}} \Bbb P( 2^{\alpha'r} |h^\epsilon(n+k2^{-r})-h^\epsilon(n+(k-1)2^{-r})| > (1+|n|)^{\delta'} a) \\ &\leq \sum_{\substack{n\in\Bbb Z\\r\in \Bbb N\\1\leq k \leq 2^r}} a^{-p} 2^{\alpha'pr}(1+|n|)^{-\delta'p}  \Bbb E\big|h^\epsilon(n+k2^{-r})-h^\epsilon(n+(k-1)2^{-r})\big|^p \\ &\leq a^{-p} \sum_{\substack{n\in\Bbb Z\\r\in \Bbb N\\1\leq k \leq 2^r}} 2^{(\alpha'-\alpha)pr} (1+|n|)^{(\delta-\delta')p} \\ &= a^{-p} \sum_{\substack{n\in\Bbb Z\\r\in \Bbb N}} 2^{\big[1+(\alpha'-\alpha)p\big]r} (1+|n|)^{(\delta-\delta')p}
\end{align*}
The series converges so long as $(\delta-\delta')p<-1$ and $1+(\alpha'-\alpha)p<0.$ We already chose $\delta'$ earlier so as to satisfy the condition $(\delta-\delta')p<-1$. Now $\alpha'$ can be chosen for instance $\frac12(\alpha-\frac1p)$ which is positive since $p>\alpha^{-1}.$
\end{proof}

\subsection{More general models and further problems}

One may ask the question of how robust the above method of proof is. The answer is that it is generalizable to more complex systems than ASEP, but it is not all-encompassing. More precisely, the method is applicable to any particle system where 
\begin{itemize}
    \item both \eqref{M} and \eqref{A} hold.
    
    \item one has a discrete martingale equation as in \eqref{discshe}.
    
    \item the discrete martingales from \eqref{discshe} satisfy \eqref{qv}.
    
\end{itemize}
Then one can essentially copy and paste the proof above (with minor modifications) to prove joint convergence in those systems as well.
\\
\\
For instance, by taking $\lambda_i=0$ in Theorem \ref{mon1}, our method will also work to show joint convergence of the nearest-neighbor \textit{symmetric} simple exclusion process to the Edwards-Wilkinson fixed point. Actually this is even simpler, as one need not perform a nonlinear transform to obtain a discrete SPDE as in \eqref{discshe}. The height function itself will satisfy an equation similar to \eqref{discshe} with the martingales satisfying \eqref{qv}. The proofs of all other propositions and lemmas work in precisely the same way as done above.
\\
\\
Less trivial examples of systems satisfying all three of the points above are higher-spin misanthrope processes. One concrete example of such a particle system is the ASEP($q,J$) model from \cite{CST18}. This comes from the generator \eqref{gen} on $\{0,...,J\}^\Bbb Z$ by taking $$b(1,a,b):=\frac1{2[J]_q}q^{a-b-(J+1)} [a]_q[J-b]_q,\;\;\;\;\;\;\; b(-1,a,b):=\frac1{2[J]_q}q^{a-b-(J+1)} [J-a]_q[b]_q,$$ where $[a]_q:=\frac{q^q-q^{-a}}{q-q^{-1}}$ for $q\in(0,1).$ ASEP($q,J$) satisfies \eqref{A} as well as \eqref{M} thanks to the nearest-neighbor interaction. The main result of \cite{CST18} then proves convergence of the associated (diffusively scaled and renormalized) height function to the KPZ equation by scaling the model parameter as $q=e^{-\epsilon}$. Note that this recovers the results of \cite{BG97} by setting $J=1$. Proposition 2.1 in \cite{CST18} says precisely that \eqref{discshe} and \eqref{qv} are satisfied with $Z_t$ defined in expression (1.8) there and $a_\epsilon, b_\epsilon$ defined accordingly in \eqref{resc}. We then have the following result:
\begin{thm}
Theorems \ref{mr} and \ref{mr2} still hold if we replace ASEP by ASEP($q,J)$, scaling $q$ as $e^{-\epsilon}$ in the model parameters above.
\end{thm}

\begin{proof}
The proof of Lemma \ref{lem1} holds essentially verbatim as given. For the proof of Lemma \ref{lem2}, we need to replace \eqref{discshe} with the appropriate modification and then verify that \eqref{qv} still holds. See equation (1.8) of \cite{CST18} for the appropriate modification of the discrete equation \eqref{discshe}, and see Proposition 2.1 of \cite{CST18} for the proof that \eqref{qv} still holds. The proof of Lemma \ref{two} still holds verbatim, since height functions which are viable for ASEP are still viable for ASEP($q,J)$ (after perhaps multiplying by 2 in the case that $J$ is even). In the proofs of Proposition \ref{prop} and Theorem \ref{mr}, the argument requires a slight modification: on the probability space $(\Omega,\mathcal F,\Bbb P)$ one should just take the Poisson clocks to be of rate one, and to account for the jump rate differences one should instead add i.i.d. uniform variables to each bond, which are independent of the Poisson clocks. The reason for this is discussed in Subsection 3.1: if $J>1$ then the construction of the basic coupling is slightly more complicated than for single-spin systems. The proof of Theorem \ref{mr2} is unchanged.
\end{proof}

Examples of interesting systems that do \textbf{not} satisfy property \eqref{M} are the non-simple exclusion processes studied for instance in \cite{DT16,Yang}. These processes have a generator similar to \eqref{gen}, the only difference is that non-neighboring sites may interact with one another, so $b$ can be a function from $\Bbb Z\times \{0,...,J\}^2\to [0,1]$ and the sum in \eqref{gen} would be over all pairs $(x,y)\in \Bbb Z^2$. Particles may jump over other particles in these systems, which locally allows height functions to overtake one another. These systems still satisfy \eqref{A}, and thus our proof still works as long as the two sequences of near-stationary initial data are coupled so that one always dominates the other (Lemma \ref{lem2}), however for arbitrary sequences one probably needs to use a different method without appealing to a black box like Theorem \ref{mon1}. For instance one can hope to directly study the quadratic variations appearing in \eqref{discshe}.
\\
\\
Then there are also open boundary systems such as those considered in \cite{CS18}. These do seem to satisfy \eqref{M} but the missing part of the argument is the boundary analogue of Theorem \ref{mon1}. Furthermore, there are also discrete-time vertex models and their degenerations, such as those studied in \cite{CT17, Gho17, CGST20, Lin20}, for which KPZ fluctuations are known. We do not know if these systems fall within the scope of our work, since the rules of their evolution are more complex and do not exactly fit the framework of the misanthrope-type exclusion processes we have described in Subsection 3.1. In particular it is unclear what exactly the basic coupling even means for these models. Some of the aforementioned systems may be explored in future work.
\\
\\
Here is another direction in which one can hope to generalize Theorem \ref{mr}. Rather than making the model more complicated, one can instead hope to strengthen the topology in which convergence occurs. Specifically one can hope to prove \textit{uniform} convergence of the entire stochastic flow of ASEP to that of the KPZ equation. More precisely, fix a compact set $K \subset \mathscr C^\alpha_\delta(\Bbb R)$ and let $K_\epsilon \subset \mathcal V^\epsilon \cap \mathscr C^\alpha_\delta$ be a sequence of compact sets that converge to $K$ in the sense of Hausdorff distance, as $\epsilon \to 0$ (where $\mathcal V^\epsilon$ was defined just before Lemma \ref{two}). Consider the random maps $\Phi^\epsilon_t:K_\epsilon \to C(\Bbb R)$ which (for a fixed realization of the Poisson clocks) sends a rescaled initial height function $h$ to the height profile at time $t$ of the ASEP profile started from $h$ and whose dynamics are run according to those Poisson clocks. Consider also the continuum version $\Phi: K\to C(\Bbb R)$ which (for a fixed realization of $\xi$) sends a function $h$ to the time $t$ solution of the KPZ equation started from $h$ and driven by $\xi$. Let $G(\Phi_t^\epsilon)$ denote the set of all $(h,\Phi_t^\epsilon(h))$ such that $h\in K_\epsilon$, and likewise let $G(\Phi_t)$ denote the set of all $(h,\Phi_t(h))$ such that $h\in K$. Also let $d_H$ denote the Hausdorff distance on compact subsets of $C(\Bbb R)\times C(\Bbb R)$ (one could also hope to use the stronger topology of $\mathscr C_\delta^\alpha(\Bbb R)\times \mathscr C_\delta^\alpha(\Bbb R)$). Then one can hope to prove convergence of the entire flow $(\Phi^\epsilon_t)_{t\in [0,T]}$ to $(\Phi_t)_{t\in[0,T]}$ where the convergence is meant to be interpreted, for instance, in the sense that (via Skorohod's representation theorem) there exists a coupling of all $\Phi_t^\epsilon, \Phi_t$ onto some probability space and some $\alpha \in (0,1/2)$ such that $$\limsup_{\delta\to 0} \limsup_{\epsilon \to 0} \bigg[\sup_{t\in [0,T]} d_H(G(\Phi_t^\epsilon),G(\Phi_t))+  \sup_{|s-t|\leq \delta} \frac{d_H(G(\Phi_t^\epsilon),G(\Phi_s^\epsilon))}{|t-s|^\alpha\vee \epsilon^\alpha} \bigg]= 0,$$ where the extra factor $\epsilon^\alpha$ in the denominator is to account for jumps. Theorem \ref{mr} and Remark \ref{>2} show (in some sense) that convergence of these flows holds in the sense of finite-dimensional distributions, but extending the convergence to this uniform Hölder sense might be more interesting. The goal would be to prove this for arbitrary compact sets $K$ and arbitrary approximating sequences $K_\epsilon.$ We do not have strong enough spatial or temporal estimates required to do this, except for the trivial case where $K,K_\epsilon$ are all finite sets with cardinality bounded in $\epsilon,$ in which case the methods of \cite{BG97} combined with our methods used to prove Theorem \ref{mr} are enough.

\end{document}